\documentclass[11pt]{amsart}
\usepackage{amssymb,amsmath,amsfonts,amsthm,mathrsfs,fullpage}
\usepackage{color,epsfig}
\numberwithin{equation}{section}

\newcommand{\FF}{\mathbb F}

\newcommand{\PP}{\mathbb P}
\newcommand{\QQ}{\mathbb Q}

\newcommand{\ZZ}{\mathbb Z} 

\newcommand{\Zhat}{\widehat\ZZ}

\newcommand{\calA}{\mathcal A}

\newcommand{\calF}{\mathcal F}

\newcommand{\calN}{\mathcal N}

\newcommand{\calS}{\mathcal S}

\newcommand{\g}{\mathfrak g}

 \def\Gal{\operatorname{Gal}}
\def\End{\operatorname{End}}

\def\cyc{{\operatorname{cyc}}}

\def\tr{\operatorname{tr}}

\def\Gal{\operatorname{Gal}}
\def\ord{\operatorname{ord}}

\def\sl{\mathfrak{sl}}
\def\gl{\mathfrak{gl}}
  
\def \GL{\operatorname{GL}}  
\def \PGL{\operatorname{PGL}}
\def \SL{\operatorname{SL}}

\def\Aut{\operatorname{Aut}} 
\def\End{\operatorname{End}}
\def\Frob{\operatorname{Frob}}
\def\lcm{\operatorname{lcm}}

\newcommand{\tors}{{\operatorname{tors}}}

\def\tr{\operatorname{tr}}

\def\lcm{\operatorname{lcm}}

 \def \Aut {\operatorname{Aut}}

\definecolor{purple}{rgb}{1,0,1}

\def\bbar#1{\setbox0=\hbox{$#1$}\dimen0=.2\ht0 \kern\dimen0 
\overline{\kern-\dimen0 #1}}
\newcommand{\Qbar}{{\overline{\mathbb Q}}}

\newcommand{\FFbar}{\overline{\FF}}

\newcommand{\defi}[1]{\textsf{#1}} 
\newtheorem{thm}{Theorem}[section]
\newtheorem{lemma}[thm]{Lemma}

\newtheorem{prop}[thm]{Proposition}

\theoremstyle{definition}

\newtheorem{conj}[thm]{Conjecture}

\theoremstyle{remark}
\newtheorem{remark}[thm]{Remark}

\newenvironment{romanenum}{\hfill \begin{enumerate} }{\end{enumerate}}

\definecolor{webcolor}{rgb}{0.8,0,0.2}
\definecolor{webbrown}{rgb}{.6,0,0}
\usepackage[
        colorlinks,
        linkcolor=webbrown,  filecolor=webcolor,  citecolor=webbrown, 
        backref,
        pdfauthor={David Zywina}, 
        pdftitle={Bounds for Serre's open image theorem},
]{hyperref}
\usepackage[alphabetic,backrefs,lite]{amsrefs} 

\begin{document}
\title{Bounds for Serre's open image theorem}

\author{David Zywina}
\email{zywina@math.upenn.edu}
\address{Department of Mathematics, University of Pennsylvania, Philadelphia, PA 19104-6395, USA}
\date{February 22, 2011}

\begin{abstract}
Let $E$ be an elliptic curve over the rationals without complex multiplication.   The absolute Galois group of $\QQ$ acts on the group of torsion points of $E$, and this action can be expressed in terms of a Galois representation $\rho_E\colon \Gal(\Qbar/\QQ) \to \GL_2(\Zhat)$.   A renowned theorem of Serre says that the image of $\rho_E$ is open, and hence has finite index, in $\GL_2(\Zhat)$.      We give the first general bounds of this index in terms of basic invariants of $E$.  For example, the index $[\GL_2(\Zhat):\rho_E(\Gal(\Qbar/\QQ))]$ can be bounded by a polynomial function of the logarithmic height of the $j$-invariant of $E$.    As an application of our bounds, we settle an open question on the average of constants arising from the Lang-Trotter conjecture.
\end{abstract}

\maketitle

\section{Introduction}

\subsection{Main theorem}
Consider an elliptic curve $E$ defined over $\QQ$ which has no complex multiplication.    For each positive integer $m$, let $E[m]$ be the group of $m$-torsion points in $E(\Qbar)$; it is a free $\ZZ/m\ZZ$-module of rank $2$.  The group $E[m]$ has an action by the absolute Galois group $\Gal_\QQ:=\Gal(\Qbar/\QQ)$ which, after a choice of basis, can be expressed in terms of a Galois representation
$\rho_{E,m} \colon \Gal_\QQ \to \Aut(E[m])\cong \GL_2(\ZZ/m\ZZ)$.   Combining these representations together, we obtain a single Galois representation
\[
\rho_E \colon \Gal_\QQ \to \Aut(E_{\tors})\cong \GL_2(\Zhat)
\]
which describes the Galois action on all the torsion points of $E$ (where $\Zhat$ is the profinite completion of $\ZZ$).  

A famous theorem of Serre \cite{MR0387283} says that $\rho_E(\Gal_\QQ)$ is open, and hence of finite index, in $\GL_2(\Zhat)$.  Serre's theorem is qualitative in nature, and it does not give an explicit bound for the index.  Our main theorem gives such a bound.

\begin{thm} \label{T:main}
Let $E$ be a non-CM elliptic curve defined over $\QQ$.    Let $j_E$ be the $j$-invariant of $E$ and let $h(j_E)$ be its logarithmic height.  Let $N$ be the product of primes for which $E$ has bad reduction and let $\omega(N)$ be the number of distinct prime divisors of $N$.
\begin{romanenum}
\item
There are absolute constants $C$ and $\gamma$ such that 
\[
\big[\GL_2(\Zhat): \rho_E(\Gal_\QQ)\big] \leq C \big( \max\{1, h(j_E)\} \big)^\gamma.
\]
\item  There is an absolute constant $C$ such that
\[
\big[\GL_2(\Zhat): \rho_E(\Gal_\QQ)\big] \leq C \Big( 68 N (1+\log\log N)^{1/2} \Big)^{24\omega(N)}.
\]
\item  Assume that the Generalized Riemann Hypothesis holds.   Then there is an absolute constant $C$ such that
\[
\big[\GL_2(\Zhat): \rho_E(\Gal_\QQ)\big] \leq  \Big( C (\log N) (\log\log(2N))^3 \Big)^{24\omega(N)}.
\]
\end{romanenum}
\end{thm}

These seem to be the first general bounds of Serre's index that hold for \emph{all} non-CM elliptic curves $E/\QQ$ (Serre and others have computed the index for several examples, and Jones \cite{Jones-AAECASC} has shown that the index is $2$ for ``most'' $E/\QQ$).  

\subsection{Background and overview}
Fix a non-CM elliptic curve $E$ over $\QQ$.   We first review the well-studied ``horizontal'' situation, that is, we will consider the groups $\rho_{E,\ell}(\Gal_\QQ)\subseteq \GL_2(\ZZ/\ell\ZZ)$ for a varying prime $\ell$.   

Suppose that $\rho_{E,\ell}$ is not surjective; its image is contained in some maximal subgroup $M$ of $\GL_2(\ZZ/\ell\ZZ)$.  If $\ell>17$ and $\ell\neq 37$, then $\rho_{E,\ell}(\Gal_\QQ)$ is contained in the normalizer of a Cartan subgroup but is not contained in any Cartan subgroup (this uses a major theorem of Mazur to rule out the case where $M$ is a Borel subgroup) \cite{MR644559}*{\S8.4}.   The case where $M$ is the normalizer of a \emph{split} Cartan subgroup has recently been ruled out by Bilu and Parent for all primes $\ell \geq c_0$ where $c_0$ is some absolute constant \cite{Bilu-Parent2011}.

Let $c(E)\in [1,+\infty]$ be the smallest number for which $\rho_{E,\ell}(\Gal_\QQ)=\GL_2(\ZZ/\ell\ZZ)$ for all primes $\ell > c(E)$.   The main task in \cite{MR0387283} is to show that $c(E)$ is finite.  Serre then asked whether the constant $c(E)$ can be bounded independent of $E$ \cite{MR0387283}*{\S4.3}, and moreover whether we always have  $c(E)\leq 37$ \cite{MR644559}*{p.~399} (this would be best possible since there are non-CM elliptic curves over $\QQ$ which have a rational isogeny of degree 37 \cite{MR0366930}).  If $c(E)$ could be bounded independent of $E$, then there would exist an absolute constant $C$ such that $[\GL_2(\Zhat): \rho_E(\Gal_\QQ)] \leq C$ (see Remark~\ref{R:uniform index}).

Under GRH, Serre proved that $c(E)\leq c \log N (\log\log 2N)^3$ for some absolute constant $c$ \cite{MR644559} (we will discuss his approach in \S\ref{S:quadratic}).   Using a variant of Serre's method, Kraus proved that $c(E) \leq 68 N (1+\log\log N)^{1/2}$ under the now superfluous assumption that $E/\QQ$ is modular \cite{MR1360773}; a similar bound was given by Cojocaru \cite{MR2118760}.    Masser and W\"ustholz have proved that $c(E)\leq C \big( \max\{1, h(j_E)\} \big)^\gamma$ where $C$ and $\gamma$ are effective absolute constants \cite{MR1209248} (see \S\ref{S:MW} for related material).     The influence of these results on this paper should be apparent from the nature of the bounds in Theorem~\ref{T:main}.\\

We now consider the ``vertical'' situation, that is, we will examine the groups $\rho_{E,\ell^n}(\Gal_\QQ)\subseteq \GL_2(\ZZ/\ell^n\ZZ)$ for a fixed prime $\ell$ and increasing $n$.   Combining these representations together, we obtain an $\ell$-adic representation
\[
 \rho_{E,\ell^\infty}\colon \Gal_\QQ\to \GL_2(\ZZ_\ell).
 \]    
In \cite{MR1484415}*{IV-11}, Serre proved that the group $G_\ell:=\rho_{E,\ell^\infty}(\Gal_\QQ)$ is open in $\GL_2(\ZZ_\ell)$.  He accomplishes this by showing that $G_\ell$, viewed as an $\ell$-adic Lie group, has the largest possible Lie algebra.     Unfortunately, this does not give a bound for the index $[\GL_2(\ZZ_\ell): G_\ell]$.

Consider a prime $\ell$ satisfying $\ell>17$ and $\ell\neq 37$.   From the horizontal setting, we know that $\rho_{E,\ell}(\Gal_\QQ)$ is either $\GL_2(\ZZ/\ell\ZZ)$ or is contained in the normalizer of a Cartan subgroup.   If $\rho_{E,\ell}(\Gal_\QQ)=\GL_2(\ZZ/\ell\ZZ)$, then one can show that $G_\ell = \GL_2(\ZZ_\ell)$; so we will focus on the second case.   The following key proposition gives constraints on the image of $\rho_{E,\ell^\infty}$; its proof is mainly group theoretic and will make up most of \S\ref{S:group theory}.

\begin{prop} \label{P:dichotomy main}
Let $E$ be an non-CM elliptic curve over $\QQ$.  Let $\ell$ be a prime for which $\ell>17$ and $\ell \neq 37$.   Then for every positive integer $n$, one of the following holds:
\begin{itemize}
\item $\rho_{E,\ell^n}(\Gal_\QQ)$ is contained in the normalizer of a Cartan subgroup of $\GL_2(\ZZ/\ell^n\ZZ)$,
\item $\rho_{E,\ell^\infty}(\Gal_\QQ) \supseteq I+\ell^{4n}M_2(\ZZ_\ell).$  
\end{itemize}
\end{prop}

We are thus led to consider the problem of effectively bounding the primes $\ell$ and positive integers $n$ for which $\rho_{E,\ell^n}(\Gal_\QQ)$ is contained in the normalizer of a Cartan subgroup.   We will use the methods of the papers mentioned in the horizontal setting which dealt with the case $n=1$.   

Finally, to obtain our bound for $[ \GL_2(\Zhat) : \rho_E(\Gal_\QQ) ]$ we will need the following result which we will prove in \S\ref{S:main proof}.

\begin{prop} \label{P:put together}  Let $E$ be an elliptic curve over $\QQ$.   Let $M$ be a positive integer with the following property: if $\ell$ is a prime greater than 17 and not 37 and $n\geq 1$ is an integer for which $\rho_{E,\ell^n}(\Gal_\QQ)$ is contained in the normalizer of a Cartan subgroup, then $\ell^n$ divides $M$.   With such an $M$, there is an absolute constant $C$ such that $[ \GL_2(\Zhat) : \rho_E(\Gal_\QQ) ] \leq C M^{24}$.
\end{prop}

For example, in Proposition~\ref{P:Serre approach}(ii) we will prove that there is an $M$ as in Proposition~\ref{P:put together} satisfying $M\leq \big( 68 N (1+\log\log N)^{1/2} \big)^{\omega(N)}$.   Theorem~\ref{T:main}(ii) then follows immediately from the above proposition.

\begin{remark}
\begin{romanenum}
\item
The constants $C$ in Theorem~\ref{T:main} are not easily computed because, in the proof of Proposition~\ref{P:put together}, we have applied Faltings theorem to several modular curves to control the contribution from the small primes.   Besides this our bounds are effective; for example it will be clear from the proof that the index $[\GL_2(\ZZ/m\ZZ): \rho_{E,m}(\Gal_\QQ)\big]$ is bounded by $( 68 N (1+\log\log N)^{1/2} )^{24\omega(N)}$ for all positive integers $m$ relatively prime to $37\prod_{\ell\leq 17}\ell$.

\item
Serre's theorem is stated more generally for non-CM elliptic curves $E$ over a number field $K$.  We expect Theorem~\ref{T:main}(i) to hold for all such $E/K$ except with a constant $C$ now depending on the field $K$.  The methods used in the proofs of Theorem~\ref{T:main}(ii) and (iii) seem unlikely to extend to a general number field $K$.
\end{romanenum}
\end{remark}

\subsection{Lang-Trotter constants on average}
We now give a quick application of Theorem~\ref{T:main}.  Fix an integer $r$ and let $E$ be an elliptic curve defined over $\QQ$.  For each prime $p$  of good reduction, we obtain an elliptic curve $E_p$ over $\FF_p$ by reduction modulo $p$.   Let $a_p(E)$  be the trace of the Frobenius automorphism of $E_p/\FF_p$; it is the integer for which $|E_p(\FF_p)| = p +1 -a_p(E)$.  We define $\pi_{E,r}(x)$ to be the number of primes $p\leq x$ of good reduction for which $a_p(E)=r$.

For example, if $r=0$, then $\pi_{E,0}(x)$ counts the number of supersingular primes $p\leq x$ of $E$ (except possibly undercounting at $p=2$ and $p=3$).   When $r=1$, the function $\pi_{E,1}(x)$ counts the number of ``anomalous'' primes of $E$ up to $x$ \cite{MR0444670}.   For the CM elliptic curve $E\colon Y^2=X^3-X$ we have $a_p(E)=\pm 2$ if and only if $p=n^2+1$, so $\pi_{E,-2}(x)+\pi_{E,2}(x)$ counts the number of primes $p\leq x$ that are of the form $n^2+1$.   The following conjecture of Lang and Trotter \cite{MR0568299} predicts the asymptotics of $\pi_{E,r}(x)$ as $x\to +\infty$.

\begin{conj}[Lang-Trotter] \label{C:LT}
Let $E$ be an elliptic curve over $\QQ$ and let $r$ be an integer.  Except for the case where $r=0$ and $E$ has complex multiplication, there is an explicit constant $C_{E,r}\geq 0$ such that 
\[
\pi_{E,r}(x) \sim C_{E,r} \frac{\sqrt{x}}{\log x}
\]
as $x\to +\infty$. 
\end{conj}

If $C_{E,r}=0$, then we interpret the conjecture as simply saying that $\pi_{E,r}(x)$ is a bounded function of $x$.  The constant $C_{E,r}$, as predicted by Lang and Trotter, is expressed in terms of the image of the representation $\rho_E$.  We will describe it for non-CM curves in \S\ref{S:LT}.    

The Lang-Trotter conjecture is open for every pair $(E,r)$ for which $C_{E,r}\neq 0$.   Furthermore, there are no known examples of $E/\QQ$ and $r\neq 0$ for which $\lim_{x\to \infty} \pi_{E,r}(x) = \infty$ (if $E$ is non-CM, then $E$ has infinitely many supersingular primes by Elkies \cite{MR903384}).\\

One source of theoretical evidence for the Lang-Trotter conjecture are ``on average'' versions of it, that is, take the average of the functions $\pi_{E,r}(x)$ over a family of elliptic curves $E$ and show that it is compatible with what one expects from the Lang-Trotter conjecture.    For real numbers $A,B\geq 2$, let $\calF(A,B)$ be the set of pairs $(a,b)\in\ZZ^2$ with $|a|\leq A$ and $|b|\leq B$ such that $4a^3+27b^2\neq 0$. For each $(a,b)\in \calF(A,B)$, let $E(a,b)$ be the elliptic curve over $\QQ$ defined by the Weierstrass equation $y^2=x^3+ax+b$.   

\begin{thm}[David-Pappalardi \cite{MR1677267}]  \label{T:DP}
Fix $\varepsilon>0$.  Let $A=A(x)$ and $B=B(x)$ be functions of $x\geq 2$ for which $A,B > x^{1+\varepsilon}$.  Then 
\[
\frac{1}{|\calF(A,B)|} \sum_{(a,b)\in \calF(A,B)} \pi_{E(a,b),r}(x) \sim C_r \frac{\sqrt{x}}{\log x} 
\]
as $x\to +\infty$, where
\[
C_r := \frac{2}{\pi} \prod_{\ell | r} \Big(1-\frac{1}{\ell^2}\Big)^{-1} \prod_{\ell\nmid r} \frac{\ell(\ell^2-\ell-1)}{(\ell-1)(\ell^2-1)}.
\]
\end{thm}

So averaging the function $\pi_{E,r}(x)$ over the family $\calF(A,B)$, with $A$ and $B$ sufficiently large in terms of $x$, we obtain a quantity with the expected order of magnitude (\cite{MR1677267} also give a version with explicit error terms).   Since conjecturally  $\pi_{E(a,b),r}(x)\sim C_{E(a,b),r} {\sqrt{x}}/{\log x}$, it is thus natural to ask (as David and Pappalardi did in \S2 of \cite{MR1677267}) whether $C_r$ is the average of the constants $C_{E(a,b),r}$.   So that the average is well-defined, we will arbitrarily define $C_{E,0}=0$ when $E$ has complex multiplication.  In \S\ref{S:LT}, we shall prove the following theorem (building off the work of N.~Jones):

\begin{thm}  \label{T:LT}
There is an absolute constant $\delta>0$ such that
\[
\frac{1}{|\calF(A,B)|} \sum_{(a,b) \in \calF(A,B)} C_{E(a,b),r} = C_r  + O\bigg(\frac{\log^\delta(AB)}{\sqrt{\min\{A, B\}}}\bigg)
\]
where the implicit constant depends only on $r$. 
\end{thm}

Thus as long as $A$ and $B$ are not that different in magnitude, the average of the constants $C_{E(a,b),r}$ over the pairs $(a,b)\in\calF(A,B)$ will be well-aproximated by $C_r$.  In particular, 
\[
\lim_{A\to + \infty} \frac{1}{|\calF(A,A^\beta)|} \sum_{(a,b) \in \calF(A,A^\beta)} C_{E(a,b),r} = C_r
\]
for any real number $\beta>0$. 

Theorem~\ref{T:LT} is a special case of a theorem of Jones \cite{MR2534114} which had the additional hypothesis that there exists an absolute constant $c$ such that $\rho_{E,\ell}(\Gal_\QQ)=\GL_2(\ZZ/\ell\ZZ)$ for all non-CM elliptic curves $E$ over $\QQ$ and primes $\ell> c$.   This extra assumption is needed by Jones to control the ``error term''; the potential problem being that a few elliptic curves with extremely large constants $C_{E,r}$ might over contribute in the average.     

We will use Theorem~\ref{T:main}(i) to control the constants $C_{E,r}$.  The connection with our result is that $C_{E,r} \leq  [\GL_2(\Zhat): \rho_E(\Gal_\QQ)]\cdot C_r$ for all non-CM elliptic curves $E/\QQ$ and integers $r$.

\subsection*{Acknowledgments}
Thanks to David Brown for his helpful comments.

\section{Group theory} \label{S:group theory}
\subsection{Cartan subgroups} \label{SS:Cartan}

Fix an odd prime $\ell$ and let $\calA$ be a free commutative \'etale $\ZZ_\ell$-algebra of rank $2$.   Up to isomorphism, there are two such algebras; we will say that $\calA$ is \defi{split} or \defi{non-split} if $\calA/\ell \calA$ is isomorphic to $\FF_\ell\times \FF_\ell$ or $\FF_{\ell^2}$, respectively.   The algebra $\calA$ has a unique automorphism $\sigma$ of order 2 which induces the natural involution on $\calA/\ell \calA$.

The algebra $\calA$ acts on itself by (left) multiplication, so a choice of $\ZZ_\ell$-basis for $\calA$ gives an embedding $\iota\colon \calA\hookrightarrow \End_{\ZZ_\ell}(\calA) \thickapprox M_2(\ZZ_\ell)$ of $\ZZ_\ell$-algebras.    Denote the image of $\iota$ by $R$.  For a positive integer $n$, the subgroup $C(\ell^n)=(R/\ell^n R)^\times$ of $\GL_2(\ZZ/\ell^n\ZZ)$ is called a \defi{Cartan subgroup} of $\GL_2(\ZZ/\ell^n\ZZ)$.   Up to conjugation in $\GL_2(\ZZ/\ell^n\ZZ)$, there are two kinds of Cartan subgroups; $C(\ell^n)$ is \defi{split} or \defi{non-split} if $R$ is \defi{split} or \defi{non-split}, respectively.    The Cartan subgroup $C(\ell^n)$ is its own centralizer in $\GL_2(\ZZ/\ell^n\ZZ)$.   We will denote the normalizer of $C(\ell^n)$ in $\GL_2(\ZZ/\ell^n\ZZ)$ by $C^+(\ell^n)$; the index of $C(\ell^n)$ in $C^+(\ell^n)$ is 2 and $\iota(\sigma)$ is a representative of the non-identity coset of $C^+(\ell^n)/C(\ell^n)$.

More concretely, the group $\{ \left(\begin{smallmatrix}a & 0 \\  0 & b\end{smallmatrix}\right) : a,b \in (\ZZ/\ell^n\ZZ)^\times \}$ is a split Cartan subgroup of $\GL_2(\ZZ/\ell^n\ZZ)$ and the matrix $\left(\begin{smallmatrix}0 & 1 \\1 & 0\end{smallmatrix}\right)$ lies in its normalizer.     With a fixed non-square $\varepsilon\in (\ZZ/\ell^n\ZZ)^\times$, the group $\{ \left(\begin{smallmatrix}a & \varepsilon b \\  b & a\end{smallmatrix}\right) : a,b \in \ZZ/\ell^n\ZZ,\, (a,b)\not\equiv (0,0) \pmod{\ell} \}$ is a non-split Cartan subgroup of $\GL_2(\ZZ/\ell^n\ZZ)$ and the matrix $\left(\begin{smallmatrix}1 & 0 \\0 & -1\end{smallmatrix}\right)$ lies in its normalizer.      

\begin{lemma} \label{L:centralizer}
Let $\alpha$ be an element of $\GL_2(\ZZ_\ell)$ for which  $\tr(\alpha)^2-4\det(\alpha)\not\equiv 0 \pmod{\ell}$, and let $R$ be the $\ZZ_\ell$-subalgebra of $M_2(\ZZ_\ell)$ generated by $\alpha$.     Take any integer $n\geq 1$ and let $\bbar\alpha$ be the image of $\alpha$ in $\GL_2(\ZZ/\ell^n\ZZ)$.
\begin{romanenum}
\item The ring $R$ is a free commutative \'etale $\ZZ_\ell$-algebra of rank $2$.  
\item
The centralizer of $\bbar\alpha$ in $\GL_2(\ZZ/\ell^n\ZZ)$ is the Cartan subgroup $C(\ell^n)=(R/\ell^n R)^\times$.
\item
If $g$ is an element of $\GL_2(\ZZ/\ell^n\ZZ)$ for which $g \bbar \alpha g^{-1}$ belongs to $C(\ell^n)$, then $g$ belongs to $C^+(\ell^n)$.
\item
Let $H$ be a cyclic subgroup of $\GL_2(\ZZ/\ell^n\ZZ)$ generated by a matrix of the form $h=I +\ell^iA$ with $1\leq i <n$.   Suppose that $H$ is stable under conjugation by $\bbar\alpha$ and that $A\pmod{\ell}$ is a non-zero element of $R/\ell R$.  Then $H\subseteq C(\ell^n)$.
\item Suppose that the image of $B\in C(\ell)$ in $\PGL_2(\ZZ/\ell\ZZ)$ has order greater than 2, then $\tr(B)^2-4\det(B)\neq0$ and $\tr(B)\neq 0$.
\end{romanenum}
\end{lemma}
\begin{proof}
\noindent (i) We claim that $\{I,\alpha\}$ is a basis of $R$ as a $\ZZ_\ell$-module; it is generated by $\{I, \alpha\}$ by the Cayley-Hamilton theorem, so it suffices to show that they are linearly independent.   Suppose that $aI + b \alpha =0$ with $a,b\in \ZZ_\ell$ not both zero.   By dividing by an appropriate power of $\ell$, we may assume that $a$ or $b$ is non-zero modulo $\ell$.  Reducing modulo $\ell$, we find that $\alpha \pmod{\ell}$ must be a scalar matrix which contradicts our assumption that $\tr(\alpha)^2-4\det(\alpha)\not\equiv 0 \pmod{\ell}$.   That $R$ is an \'etale $\ZZ_\ell$-algebra follows from $\tr(\alpha)^2-4\det(\alpha) \in \ZZ_\ell^\times$. 

\noindent (ii)-(iii)  Fix an element $m\in C^+(\ell^n)-C(\ell^n)$.   Since $\tr(\bbar\alpha)^2-4\det(\bbar\alpha)\in (\ZZ/\ell^n\ZZ)^\times$ and $\ell$ is odd, one finds that there are exactly two roots of the polynomial $x^2 - \tr(\bbar\alpha)x + \det(\bbar\alpha)$ in $R/\ell^nR$; they are $\bbar\alpha$ and $m^{-1}\bbar\alpha m$.   If $g \bbar \alpha g^{-1} = \bbar\alpha$, then $g$ commutes with $C(\ell)=(R/\ell^n R)^\times$ and hence $g\in C(\ell^n)$.  If $g \bbar \alpha g^{-1}=m^{-1}\bbar\alpha m$, then $mg$ commutes with $C(\ell^n)$ and hence $g\in C^+(\ell^n)$.

\noindent (iv)  
The group $H$ is cyclic of order $\ell^{n-i}$.  Since $\bbar\alpha h \bbar\alpha^{-1}$ is also a generator of $H$, there is a unique $m \in  (\ZZ/\ell^{n-i}\ZZ)^\times$ for which $\bbar\alpha h \bbar\alpha^{-1} = h^m$.  We have $\bbar\alpha A \bbar\alpha^{-1} \equiv A \pmod{\ell}$ since by assumption $A \pmod{\ell}$ lies in $R/\ell R$.  Therefore,
\[
I +\ell^i A\equiv \bbar\alpha h \bbar\alpha^{-1} = h^m \equiv I +\ell^i mA \pmod{\ell^{i+1}}
\]
and so $m\equiv 1 \pmod{\ell}$ since $A \pmod{\ell}$ is non-zero.   Conjugation by $\bbar\alpha$ thus gives a group automorphism of $H$ with order dividing $\ell^{n-i-1}$, and hence $\beta:=\bbar \alpha^{\ell^{n-i-1}}$ commutes with $H$.     The order of $\bbar\alpha\pmod{\ell}$ is relatively prime to $\ell$, so $\tr(\beta)^2-4\det(\beta) \not\equiv 0 \pmod{\ell}$.    Using part (ii), the centralizer of $\beta$ in $\GL_2(\ZZ/\ell^n\ZZ)$ is a Cartan subgroup that contains $C(\ell^n)$ (the centralizer of $\bbar\alpha$) thus $C(\ell^n)$ is the centralizer of $\beta$.  Therefore, $H\subseteq C(\ell^n)$. 

\noindent (v)  
That $B$ belongs to $C(\ell)$ implies that it is diagonalizable in $\GL_2(\FFbar_\ell)$.  It is then easy to show that $B^2$ is a scalar matrix (equivalently, the image of $B$ in $\PGL_2(\FF_\ell)$ has order 1 or 2) if and only if $\tr(B)^2-4\det(B)=0$ or $\tr(B)=0$.
\end{proof}

\subsection{Congruence filtration} \label{SS:filtrations}
Fix a prime $\ell$ and a closed subgroup $G$ of $\GL_2(\ZZ_\ell)$.  [We will eventually consider $G=\rho_{E,\ell^\infty}(\Gal_\QQ)$ for a non-CM elliptic curve $E/\QQ$.]  For each positive integer $n$, let $G(\ell^n)$ be the image of $G$ under the reduction modulo $\ell^n$ homomorphism $\GL_2(\ZZ_\ell)\to \GL_2(\ZZ/\ell^n\ZZ)$.  

For each $n\geq 1$, we let $G_n$ be the subgroup consisting of those $A\in G$ for which $A\equiv I \pmod{\ell^n}$.     The groups $G_n$ are normal closed subgroups of $G$ and form a fundamental system of neighbourhoods of $1$ in $G$.  

Let $\gl_2(\FF_\ell)$ be the additive group $M_2(\FF_\ell)$ and let $\sl_2(\FF_\ell)$ be the subgroup of trace 0 matrices; they are Lie algebras over $\FF_\ell$ when equipped with the usual pairing $[A,B]=AB-BA$.  For each $n\geq 1$, we have an injective group homomorphism
\begin{equation*} 
G_{n}/G_{n+1} \hookrightarrow \gl_2(\FF_\ell), \quad  I + \ell^n B  \mapsto B \bmod{\ell}
\end{equation*}
whose image we will denote by $\g_n$.    From the groups $\g_n$ and $G(\ell)$, we can recover the cardinality of each $G(\ell^i)$, though not necessarily the group.  Indeed, for each $i\geq 1$ we have $|G(\ell^{i+1})|=|G(\ell^i)| |\g_i|$, so $|G(\ell^i)|=|G(\ell)| \prod_{n=1}^{i-1}|\g_n|$.   If $\det(G_n) \subseteq 1 + \ell^{n+1} \ZZ_\ell$, then $\g_n \subseteq \sl_2(\FF_\ell)$.  In particular, if $\det(G)=1$, then $\g_n\subseteq \sl_2(\FF_\ell)$ for all $n\geq 1$.

Let $[G,G]$ be the \defi{commutator subgroup} of $G$, that is,  the smallest closed normal subgroup of $G$ whose quotient group is abelian.

\begin{lemma} \label{L:filtration basics}  Fix an integer $n\geq 1$ (assume that $n\geq 2$ if $\ell=2$).
\begin{romanenum}
\item  $\g_n \subseteq \g_{n+1}$.
\item  If $\g_n=\gl_2(\FF_\ell)$, then $G\supseteq I+\ell^n M_2(\ZZ_\ell)$.
\item  If $\det(G)=1$ and $\g_n=\sl_2(\FF_\ell)$, then $G\supseteq \{ A \in \SL_2(\ZZ_\ell) : A\equiv I \pmod{\ell^n} \}$. 
\item  If $\g_n= \g_{2n}$, then $\g_n$ is a Lie subalgebra of $\gl_2(\FF_\ell)$.
\item  Let $e=0$ or $1$ if $\ell$ is odd or even, respectively.   Suppose that $G(\ell^{n+1+e})\supseteq \{A \in \SL_2(\ZZ/\ell^{n+1+e}\ZZ) : A \equiv I \pmod{\ell^n}\}$, then $[G,G]$ contains $\{ A \in \SL_2(\ZZ_\ell) : A\equiv I \pmod{\ell^{2n+e}} \}$.
\end{romanenum}
\end{lemma}
\begin{proof}
\noindent (i) Take any $B\in \g_n$.  There exists an $A\in M_2(\ZZ_\ell)$ such that $A\equiv B \pmod{\ell}$ and $I+\ell^n A \in G$.  Taking the $\ell$-th power of $I+\ell^n A$ we find that
\[
(I+\ell^n A)^\ell \equiv I + \ell^{n+1} A + {\ell \choose 2} \ell^{2n} A^2 \quad\pmod{\ell^{3n}}.
\]
When $\ell$ is odd, we have $(I+\ell^n A)^\ell \equiv I + \ell^{n+1} A \pmod{\ell^{n+2}}$ and hence $B\equiv A \pmod{\ell}$ belongs to $\g_{n+1}$ (when $\ell=2$ we need $n\geq 2$ to guarantee that $(I+\ell^n A)^\ell \equiv I + \ell^{n+1} A \pmod{\ell^{n+2}}$).

\noindent (ii)  We always have an inclusion $\g_i \subseteq \gl_2(\FF_\ell)$, so by part (i) and our assumption on $\g_n$ we deduce that $\g_i = \gl_2(\FF_\ell)$ for all $i\geq n$.  So for all $i\geq n$, we have $|G(\ell^i)| = |G(\ell^n)| \ell^{4(i-n)}$.   Since $G$ is a closed subgroup of $\GL_2(\ZZ_\ell)$, this is equivalent to $G$ containing the group $I+\ell^n M_2(\ZZ_\ell)$.  

\noindent (iii) The proof is similar to (ii), one shows that $|G(\ell^i)| = |G(\ell^n)| \ell^{3(i-n)}$ for all $i\geq n$.

\noindent (iv)  
Take any $B_1,B_2 \in \g_n$.  There are $A_i \in M_2(\ZZ_\ell)$ such that $A_i\equiv B_i \pmod{\ell}$ and $g_i:=I+\ell^n A_i$ belongs to $G$.    The commutator $g_1 g_2 g_1^{-1} g_2^{-1}$ belongs to $[G,G]\subseteq G$ and equals
\begin{align*}
&(I+\ell^n A_1)(I+\ell^n A_2)(I+\ell^n A_1)^{-1}(I+\ell^n A_2)^{-1}\\
=& \big( (I+\ell^n A_2)(I+\ell^n A_1) +\ell^{2n}(A_1 A_2 - A_2A_1) \big) (I+\ell^n A_1)^{-1}(I+\ell^n A_2)^{-1}\\
=& I + \ell^{2n}(A_1 A_2 - A_2A_1) (I+\ell^n A_1)^{-1}(I+\ell^n A_2)^{-1}\\
\equiv & I + \ell^{2n}(A_1 A_2 - A_2A_1) \quad \pmod{\ell^{3n}}.
\end{align*}
Therefore, $[B_1,B_2] = B_1B_2-B_2B_1 \equiv A_1 A_2 - A_2A_1 \pmod{\ell}$ is an element of $\g_{2n}$.   Since $\g_n=\g_{2n}$ by assumption, we deduce that $\g_n$ is closed under the Lie bracket $[\cdot,\cdot]$.  

\noindent (v)
Set $S:=[G,G]$ and let $\{\mathfrak{s}_i\}_{i\geq 1}$ be the filtration of $S$; each space $\mathfrak{s}_n$ is contained in $\sl_2(\FF_\ell)$.   We claim that $\mathfrak{s}_{2n+e}=\sl_2(\FF_\ell)$.  Once this is known, we can deduce the desired result from (iii).

Define integral matrices $B_1 = \left(\begin{smallmatrix}0 & 1 \\0 & 0\end{smallmatrix}\right)$, $B_2 =  \left(\begin{smallmatrix}0 & 0 \\1 & 0\end{smallmatrix}\right)$ and $B_3= \left(\begin{smallmatrix}1 & 0 \\0 & -1\end{smallmatrix}\right)$.  We have $\det(I+\ell^n B_i)\equiv 1 \pmod{\ell^{n+1+e}}$, so by assumption there is a $g_i\in G_n$ such that $g_i\equiv I + \ell^n B_i \pmod{\ell^{n+1+e}}$.  From the commutator calculations above, the matrix $h:=g_i g_j g_i^{-1} g_j^{-1}$ belongs to $[G,G]\cap (I+\ell^{2n}M_2(\ZZ_\ell))$ and
$h\equiv I +\ell^{2n}[B_i,B_j] \pmod{\ell^{2n+1+e}}$.  Observe that $[B_1,B_2]= B_3$, $[B_2,B_3]=2B_2$ and $[B_3,B_1]=2B_1$.  The matrices $2B_1$, $2B_2$ and $B_3$ modulo $\ell$ are thus in $\mathfrak{s}_{2n}$ and they generate $\sl_2(\FF_\ell)$ if $\ell$ is odd.  Therefore, $\mathfrak{s}_{2n}=\sl_2(\FF_\ell)$ for $\ell$ odd.   Now consider $\ell=2$.  The group $[G(\ell^{2n+2}),G(\ell^{2n+2})]$ contains $I+\ell^{2n}2B_1= I+\ell^{2n+1}B_1 \pmod{\ell^{2n+2}}$, $I+\ell^{2n}2B_2= I+\ell^{2n+1}B_2 \pmod{\ell^{2n+2}}$ and $(I+\ell^{2n}B_3)^\ell \equiv I +\ell^{2n+1} B_3 \pmod{\ell^{2n+2}}.$  Therefore, $B_1,B_2$ and $B_3$ modulo $\ell$ belong to $\mathfrak{s}_{2n+1}$ and hence $\mathfrak{s}_{2n+1}=\sl_2(\FF_\ell).$
\end{proof}

\subsection{Group theory for Proposition~\ref{P:dichotomy main}} \label{SS:dichotomy proof}
Let $G$ be a closed subgroup of $\GL_2(\ZZ_\ell)$ where $\ell$ is an odd prime, and we keep the notation from \S\ref{SS:filtrations}.  We now impose the following additional assumptions on $G$, which will hold for the rest of this section:
\begin{itemize}
\item $\det(G) = \ZZ_\ell^\times$,
\item $G(\ell)\subseteq \GL_2(\ZZ/\ell\ZZ)$ is contained in the normalizer of a Cartan subgroup, but is not contained in any Cartan subgroup,
\item the image of $G(\ell)$ in $\PGL_2(\ZZ/\ell\ZZ)$ contains an element with order at least 5. 
\end{itemize}

The goal of this section is to prove the following:
\begin{prop} \label{P:dichotomy}
For $G$ as above, for each $n\geq 1$, one of the following properties hold:
\begin{itemize}
\item $G(\ell^n)$ is contained in the normalizer of a Cartan subgroup,
\item $G\supseteq I+ \ell^{4n} M_2(\ZZ_\ell)$.
\end{itemize}
\end{prop}

Let $C(\ell)$ be a Cartan subgroup of $\GL_2(\ZZ/\ell\ZZ)$ such that $G(\ell)\subseteq C^+(\ell)$.    Using $[C^+(\ell):C(\ell)]=2$ and our assumption that $G(\ell)$ contains an element of order at least 5, we deduce from Lemma~\ref{L:centralizer}(v) that there exists an $\alpha \in G$ such that $\alpha \pmod{\ell}$ belongs to $C(\ell)$, $\tr(\alpha)^2-4\det(\alpha) \not\equiv 0 \pmod{\ell}$ and $\tr(\alpha)\not\equiv 0 \pmod{\ell}$.

Let $R$ be the $\ZZ_\ell$-subalgebra of $M_2(\ZZ_\ell)$ generated by $\alpha$, and for each $n\geq 1$ define $C(\ell^n)=(R/\ell^n R)^\times$.  By Lemma~\ref{L:centralizer}, $R$ is a free commutative \'etale $\ZZ_\ell$-algebra of rank $2$ and $C(\ell^n)$ is the unique Cartan subgroup of $\GL_2(\ZZ/\ell^n\ZZ)$ containing $\alpha$ (in particular, our two Cartan subgroups $C(\ell)$ are the same). 

The group $G$ acts on each $G_n$ by conjugation, and hence also acts on the quotient group $G_n/G_{n+1}$ with the subgroup $G_1$ acting trivially.   Therefore, $\g_n\subseteq \gl_2(\FF_\ell)$ is stable under conjugation by  $G(\ell)=G/G_1$.  To figure out the possibilities for $\g_n$, we will first decompose $\gl_2(\FF_\ell)$ into irreducible $\FF_\ell[G(\ell)]$-modules (the representation theory is straightforward since $\ell \nmid |G(\ell)|$).

\begin{lemma} \label{L:easy rep}
\begin{romanenum}
\item
As an $\FF_\ell[G(\ell)]$-module, with $G(\ell)$ acting by conjugation, $\gl_2(\FF_\ell)$ equals $V_1 \oplus V_2 \oplus V_3$ for non-isomorphic irreducible $\FF_\ell[G(\ell)]$-modules $V_i$.  They can be ordered so that $V_1=\FF_\ell\cdot I$, $V_2= \{A \in R/\ell R : \tr(A)=0\}$ and $\dim_{\FF_\ell} V_3 =2$.
\item
The space $V_3$ in (i) is not a Lie subalgebra of $\gl_2(\FF_\ell)$.
\item 
The group $\sl_2(\FF_\ell)$ is generated by the set $\{ g B g^{-1} - B : g\in G(\ell),\, B \in \gl_2(\FF_\ell) \}$.
\end{romanenum}
\end{lemma}
\begin{proof}
(i)  It is easy to check that the $V_1$ and $V_2$ as given in the last line are stable under conjugation by $C^+(\ell)$, and hence also by $G(\ell)$, and thus $\gl_2(\FF_\ell) = V_1 \oplus V_2 \oplus V_3$ for some 2-dimensional representation $V_3$.  
The $G(\ell)$-action on $V_1$ is trivial while the action on $V_2$ is non-trivial (since each element of $C^+(\ell)-C(\ell)$ acts as $-I$ on $V_2$, and by assumption we have $G(\ell)\subseteq C^+(\ell)$ and $G(\ell)\not\subseteq C(\ell)$).   

It remains to show that $V_3$ is an irreducible $\FF_\ell[G(\ell)]$-module.   Conjugation gives a faithful action of  $G(\ell)/ G(\ell)\cap \FF_\ell^\times$ on $\gl_2(\FF_\ell)$.   Since $V_1$ and $V_2$ are one dimensional, to prove that $V_3$ is an irreducible $\FF_\ell[G(\ell)]$-module, it suffices to show that the group $G(\ell)/ G(\ell)\cap \FF_\ell^\times$ is non-abelian.   

Let $\bbar\alpha$ be the image of $\alpha$ in $G(\ell)$ and let $m$ be an element in $G(\ell)-C(\ell)$. If the cosets of $m$ and $\bbar\alpha$ in $G(\ell)/ G(\ell)\cap \FF_\ell^\times$ commuted, then $m \bbar\alpha m^{-1} \bbar\alpha^{-1} = \zeta$ for some $\zeta \in \FF_\ell^\times$.   We have $\zeta\neq 1$, since $m\not\in C(\ell)$ and $C(\ell)$ is the centralizer of $\bbar\alpha$ by Lemma~\ref{L:centralizer}(ii).   So $\bbar\alpha$ and $m\bbar\alpha m^{-1}=\zeta \bbar\alpha$ are the two distinct roots of $x^2- \tr(\bbar\alpha) x+ \det(\bbar\alpha)$ in $R/\ell R$.  Therefore, $x^2- \tr(\bbar\alpha)x + \det(\bbar\alpha)=(x-\bbar\alpha)(x-\zeta \bbar\alpha)$ and hence $\tr(\alpha)=(1+\zeta) \bbar\alpha$.    Since $\tr(\bbar\alpha)\in \FF_\ell$ is non-zero by our choice of $\alpha$, we deduce that $\bbar\alpha$ is a scalar matrix.   However, this contradicts $\tr(\bbar\alpha)^2-4\det(\bbar\alpha) \neq 0$.  So as desired, the group $G(\ell)/ G(\ell)\cap \FF_\ell^\times$ is non-abelian.

\noindent (ii)
First suppose that $R$ is split.  After conjugating $R/\ell R$ by an element of $\GL_2(\ZZ/\ell\ZZ)$, we may assume that $R=\{\left(\begin{smallmatrix}a & 0 \\0 & b\end{smallmatrix}\right) \}$.  We then have $V_3= \FF_\ell \left(\begin{smallmatrix}0 & 1 \\0 & 0\end{smallmatrix}\right) \oplus \FF_\ell \left(\begin{smallmatrix}0 & 0 \\1 & 0\end{smallmatrix}\right)$, but $\big[\left(\begin{smallmatrix}0 & 1 \\0 & 0\end{smallmatrix}\right) , \left(\begin{smallmatrix}0 & 0 \\1 & 0\end{smallmatrix}\right)\big] = \left(\begin{smallmatrix}1 & 0 \\0 & -1\end{smallmatrix}\right)$ does not belong to $V_3$.  Now suppose that $R$ is non-split.  After conjugating $R/\ell R$ by an element of $\GL_2(\ZZ/\ell\ZZ)$, we may assume that $R=\{\left(\begin{smallmatrix}a & \varepsilon b \\b & a\end{smallmatrix}\right) \}$ where $\varepsilon \in (\ZZ/\ell\ZZ)^\times$ is a fixed non-square.  Then $V_3= \FF_\ell \left(\begin{smallmatrix}1 & 0 \\0 & -1\end{smallmatrix}\right) \oplus \FF_\ell \left(\begin{smallmatrix}0 & \varepsilon \\-1 & 0\end{smallmatrix}\right)$, but $\big[ \left(\begin{smallmatrix}1 & 0 \\0 & -1\end{smallmatrix}\right) , \left(\begin{smallmatrix}0 & \varepsilon \\-1 & 0\end{smallmatrix}\right)\big] = \left(\begin{smallmatrix}0 & \varepsilon \\1 & 0\end{smallmatrix}\right)$ does not belong to $V_3$.

\noindent (iii)
Let $W$ be the subspace of $\sl_2(\FF_\ell)$ generated by the set $\{ g B g^{-1} - B : g\in G(\ell),\, B \in \gl_2(\FF_\ell) \}$; each element of the set does indeed have trace 0.   The space $W$ is stable under conjugation by $G(\ell)$, so by (i) we find that $W$ is either $0$, $V_2$, $V_3$ or $V_2\oplus V_3$.  It thus suffices to show that $W$ contain non-zero elements from $V_2$ and $V_3$.   Fix $i\in\{2,3\}$.  Since $V_i$ has a non-trivial $G(\ell)$-action, there are $v\in V_i$ and $g\in G(\ell)$ such that $gvg^{-1}\neq v$.  Therefore, $gvg^{-1} -v $ is a non-zero element of $W$ that belongs to $V_i$.
\end{proof}

Define the closed subgroup $S=G\cap \SL_2(\ZZ_\ell)$ of $G$.   For each $n\geq 1$, we let $S_n=G_n\cap \SL_2(\ZZ_\ell)$.  We denote the image of the injective homomorphism 
\begin{equation} \label{E: sl2 basic}
S_{n}/S_{n+1} \hookrightarrow \gl_2(\FF_\ell), \quad  1 + \ell^n B  \mapsto B \bmod{\ell}
\end{equation}
by $\mathfrak{s}_n$, it is a subgroup of $\sl_2(\FF_\ell)$.  

\begin{lemma} \label{L:trace}
With notation as above, $\tr(\g_n)=\FF_\ell$ and $\mathfrak{s}_n = \g_n \cap \sl_2(\FF_\ell)$ for all $n\geq 1$.
\end{lemma}
\begin{proof}
First of all, we claim that $\det(G_n)=1+\ell^n\ZZ_\ell$ for all $n\geq 1$.  It suffices to show that there exists a $g\in G_1$ with $\det(g)\not\equiv 1 \pmod{\ell^2}$, since then $g^{\ell^{n-1}}\in G_n$ and one can then show that $\det(g^{\ell^{n-1}})$  generates $1+\ell^n\ZZ_\ell$ as a topological group.  Since $\det(G)=\ZZ_\ell^\times$, there exists a $g \in G$ such that $\det(g) \equiv 1 \pmod{\ell}$ and $\det(g) \not\equiv 1 \pmod{\ell^2}$.  Since $\ell \nmid |G(\ell)|$, replacing $g$ by $g^{|G(\ell)|}$, we have $g\equiv I \pmod{\ell}$ and $\det(g) \not\equiv 1 \pmod{\ell^2}$ as desired.   

For an element $g\in G_n$ of the form $g = I +\ell^n A$, the condition that $\det(g)\not\equiv 1 \pmod{\ell^{n+1}}$ is equivalent to $\tr(A)\not\equiv 0 \pmod{\ell}$.   So $\det(G_n)=1+\ell^n\ZZ_\ell$ implies that $\tr(\g_n)=\FF_\ell$.

We certainly have $\mathfrak{s}_n \subseteq \g_n \cap \sl_2(\FF_\ell)$, so we need only prove the other inclusion.   Take any $B \in \g_n$ with trace $0$, and pick an $A$ such that $I+\ell^n A \in G$ with $A\equiv B \pmod{\ell}$.  We have $\det(I+\ell^nA)\equiv 1 \pmod{\ell^{n+1}}$ since $\tr(A)\equiv 0 \pmod{\ell}$, so there exists an element $g\in G_{n+1}$ such that $\det(g)=\det(I+\ell^nA)$.    The matrix $s:=g^{-1}(I+\ell^nA)$ has determinant 1, and hence belongs to $S$.  It satisfies $s \equiv I + \ell^n A \pmod{\ell^{n+1}}$, so $B\equiv A \pmod{\ell}$ belongs to $\mathfrak{s}_n$.
\end{proof}

\begin{proof}[Proof of Proposition~\ref{P:dichotomy}] {\color{white}.}

\noindent \textbf{Case 1:} $\dim_{\FF_\ell} \g_n = 4$.

We have $\g_n=\gl_2(\FF_\ell)$, so $G$ contains the group $1+\ell^n M_2(\ZZ_\ell)$ by Lemma~\ref{L:filtration basics}(ii).

\noindent \textbf{Case 2:} $\dim_{\FF_\ell} \g_n = 3$.

If $\dim_{\FF_\ell} \g_{2n} = 4$, then Case 1 shows that $G$ contains the group $1+\ell^{2n} M_2(\ZZ_\ell)$.   So we may assume that $\dim_{\FF_\ell} \g_{2n} =3$ and hence $\g_n=\g_{2n}$.  Since $\g_n$ has dimension 3 over $\FF_\ell$, is stable under conjugation by $G(\ell)$ and satisfies $\tr(\g_n)=\FF_\ell$, we deduce from Lemma~\ref{L:easy rep}(i) that $\g_n = \FF_\ell \left(\begin{smallmatrix}1 & 0 \\0 & 1\end{smallmatrix}\right) \oplus V_3$.   We thus have $\mathfrak{s}_n = V_3$ which is \emph{not} a Lie subalgebra of $\gl_2(\FF_\ell)$ by Lemma~\ref{L:easy rep}(ii).    However, $\g_n=\g_{2n}$ and Lemma~\ref{L:filtration basics}(iv) implies that $\g_n$, and hence also $\mathfrak{s}_n = \g_n \cap \gl_2(\FF_\ell)$, is a Lie subalgebra of $\gl_2(\FF_\ell)$.  This is a contradiction, so we cannot have $\dim_{\FF_\ell} \g_{2n} =3$.
 
\noindent \textbf{Case 3:} $\dim_{\FF_\ell} \g_n = 2$.

If $\dim_{\FF_\ell} \g_{2n}$ equals 3 or 4, then Cases 1 and 2 above show that $G$ contains the group $1+\ell^{4n} M_2(\ZZ_\ell)$.   So assume that $\dim_{\FF_\ell} \g_{2n} =2$.  

Let $S(\ell^{2n})$ be the image of $S$ under the reduction modulo $\ell^{2n}$ map $\SL_2(\ZZ_\ell)\to\SL_2(\ZZ/\ell^{2n}\ZZ)$.  Let $H$ be the subgroup of $S(\ell^{2n})$ consisting of those matrices congruence to $I$ modulo $\ell^n$.     Since $\dim_{\FF_\ell} \g_n = \dim_{\FF_\ell} \g_{2n} = 2$, we deduce from Lemma~\ref{L:trace} and Lemma~\ref{L:filtration basics}(i) that $\dim_{\FF_\ell} \mathfrak{s}_i = 1$ for $n\leq i \leq 2n$.  There exists an element $h \in S(\ell^{2n})$ of the form $I+\ell^nA$ with $A \bmod{\ell}$ non-zero.  By cardinality considerations, the group $H$ is generated by $h$.   The group $H$ is normalized by $G(\ell^{2n})$.

Using $\mathfrak{s}_n\subseteq\sl_2(\FF_\ell)$, $\dim_{\FF_\ell} \mathfrak{s}_n = 1$, and Lemma~\ref{L:easy rep}(i), we find that there is only one possibility for $\mathfrak{s}_n$, i.e., $\mathfrak{s}_n= \{B\in R/\ell R : \tr(B)=0\}$.  Since $A\pmod{\ell}$ belongs to $R/\ell R$, we have $H\subseteq C(\ell^{2n})$ by Lemma~\ref{L:centralizer}(iv).  We then have
\[
H= \big\{ a \in C(\ell^{2n}) : \det(a)\equiv 1 \pmod{\ell^{2n}} \text{ and } a\equiv I \pmod{\ell^n}\big\};
\]
the inclusion ``$\subseteq$'' is now clear and both groups have cardinality $\ell^n$.   The group $G(\ell^{2n})$ normalizes $H$, hence it also normalizes $H':=\{ a \in C(\ell^{2n}) :  a\equiv I \pmod{\ell^n}\}=I+\ell^n \,R/\ell^{2n} R$ (the group $H'$ is obtained from $H$ by including scalar matrices that are congruent to the identity modulo $\ell^n$).   

Now take any $g\in G$.   For any $a\in R$, we have just shown that there exists an element $b\in R$ such that $g (I+\ell^n a) g^{-1} \equiv I + \ell^n b \pmod{\ell^{2n}}$, and hence $g a g^{-1}\equiv b \pmod{\ell^n}$.   Therefore, $g \pmod{\ell^n}$ normalizes $R/\ell^n R$ and thus belongs to $C^+(\ell^n)$.  We conclude that $G(\ell^n)\subseteq C^+(\ell^n)$.

\noindent \textbf{Case 4:} $\dim_{\FF_\ell} \g_n = 1$.

 By Lemma~\ref{L:trace} and  Lemma~\ref{L:easy rep}, we have $\g_1=\cdots = \g_n = \FF_\ell \left(\begin{smallmatrix}1 & 0 \\ 0 & 1 \end{smallmatrix}\right)$.  Let $H$ be the subgroup of $G(\ell^n)$ consisting of those matrices that are congruent to $I$ modulo $\ell$.   There exists an element $h \in H$ of the form $I+\ell A$ for which $A\equiv I \pmod{\ell}$.  By cardinality considerations, the group $H$ is generated by $h$.   The group $H$ is stable under conjugation by $G(\ell^n)$.  Therefore by Lemma~\ref{L:centralizer}(iv), we have $H\subseteq C(\ell^n)$.
 
Let $W$ be the group of $g\in G(\ell^n)$ for which $g \pmod{\ell} \in C(\ell)$, it is an index 2 subgroup of $G(\ell^n)$.   Take any $g\in W$.   Since $\alpha \pmod{\ell}$ and $g \pmod{\ell}$ commute, the commutator $g\alpha g^{-1} \alpha^{-1} \pmod{\ell^n}$ belongs to $H\subseteq C(\ell^n)$ and hence  $g\alpha g^{-1} \pmod{\ell^n}$ is an element of $C(\ell^n)$.  By Lemma~\ref{L:centralizer}(iii), this implies that $g$ is an element of $C^+(\ell^n)$.   Since $g$ was arbitrary, we have $W\subseteq C^+(\ell^n)$ and by considering its image modulo $\ell$ we must have $W\subseteq C(\ell^n)$.  Finally, the group $W\subseteq C(\ell^n)$ contains $\alpha$ and is normalized by $G(\ell^n)$, thus $G(\ell^n)\subseteq C^+(\ell^n)$ by Lemma~\ref{L:centralizer}(iii).
\end{proof}

\subsection{Proof of Proposition~\ref{P:dichotomy main}}

We will apply Proposition~\ref{P:dichotomy} with the group $G:=\rho_{E,\ell^\infty}(\Gal_\QQ)$ contained in $\GL_2(\ZZ_\ell)$.  The representation $\det\circ \rho_{E,\ell^\infty}\colon \Gal_\QQ \to \ZZ_\ell^\times$ is the $\ell$-adic cyclotomic character, and hence is surjective.  Therefore, $\det(G) = \ZZ_\ell^\times$.      If $G(\ell)=\rho_{E,\ell}(\Gal_\QQ)$ equals $\GL_2(\ZZ/\ell\ZZ)$, then $G=\GL_2(\ZZ_\ell)$ by \cite{MR644559}*{Lemme~15} and the result is immediate for all $n$.   So assume that $G(\ell)\neq \GL_2(\ZZ/\ell \ZZ)$, and hence by our restriction to $\ell>17$ and $\ell\neq 37$, the group $G(\ell)$ lies in the normalizer of some Cartan subgroup of $\GL_2(\ZZ/\ell\ZZ)$ but is not contained in any Cartan subgroup \cite{MR644559}*{Lemme~16--18}.   The image of $G(\ell)$ in $\PGL_2(\ZZ/\ell\ZZ)$ contains an element of order greater than $5$ by \cite{MR644559}*{Lemme~18'} and our assumption $\ell > 17$.

We have verified that our $G$ satisfies the assumptions of Proposition~\ref{P:dichotomy}, so for each $n\geq 1$ the group $G(\ell^n)=\rho_{E,\ell^n}(\Gal_\QQ)$ is contained in the normalizer of a non-split Cartan subgroup \emph{or} $G$ contains $I+\ell^{4n}M_2(\ZZ_\ell)$.   This completes the proof of Proposition~\ref{P:dichotomy main}.\\

To obtain better bounds, we will use the following lemma with Proposition~\ref{P:dichotomy main}.

\begin{lemma} \label{L:dichotomy supplement}
Let $E$ be a non-CM elliptic curve over $\QQ$ and let $\ell$ be a prime greater than 17 and not equal to 37.   If $\rho_{E,\ell^\infty}(\Gal_\QQ) \supseteq I +\ell^n M_2(\ZZ_\ell)$, then $\rho_{E,\ell^\infty}(\Gal(\Qbar/\QQ^\cyc)) \supseteq \{A \in \SL_2(\ZZ_\ell): A \equiv I \pmod{\ell^n}\}$
where $\QQ^\cyc$ is the cyclotomic extension of $\QQ$.
\end{lemma}
\begin{proof}
Define $G:=\rho_{E,\ell^\infty}(\Gal_\QQ)$; as noted above, $G$ satisfies the conditions of \S\ref{SS:dichotomy proof}.   Since $\QQ^\cyc$ is the maximal abelian extension of $\QQ$, we have $[G,G]=\rho_{E,\ell^\infty}(\Gal(\Qbar/\QQ^\cyc))$.  Define $S:=[G,G]\subseteq \SL_2(\ZZ_\ell)$ and for each $n\geq 1$, define $S_n$ and $\mathfrak{s}_n$ as done at (\ref{E: sl2 basic}).   The commutator map $G\times G_n \to S_n$,  $(g, I+\ell^n A)\mapsto g(I+\ell^n A) g^{-1}(I+\ell^n A)^{-1}$ induces a function
\[
f\colon G(\ell) \times \g_n \to \mathfrak{s}_n, \quad (g, B) \mapsto g B g^{-1} - B.
\]
Our hypothesis on the image of $\rho_{E,\ell^\infty}$ implies that $\g_n=\gl_2(\FF_\ell)$.  Lemma~\ref{L:easy rep}(iii) implies that $\mathfrak{s}_n=\sl_2(\FF_\ell)$ and hence $S\supseteq \{A \in \SL_2(\ZZ_\ell): A \equiv I \pmod{\ell^n}\}$ by Lemma~\ref{L:filtration basics}(iii).
\end{proof}

\section{Quadratic characters arising from non-surjective $\rho_{E,\ell}$} \label{S:quadratic}

Let $E$ be a non-CM elliptic curve defined over $\QQ$.  Let $N$ be the product of primes $p$ for which $E$ has bad reduction and let $\omega(N)$ be the number of distinct prime factors of $N$.   We now follow an approach used by Serre \cite{MR644559}*{\S8.4}, and then by Kraus \cite{MR1360773} and Cojocaru \cite{MR2118760}.

Let $\ell$ be a prime satisfying $\ell>17$ and $\ell\neq 37$ for which $\rho_{E,\ell}(\Gal_\QQ)\neq \GL_2(\ZZ/\ell\ZZ)$.   By \cite{MR644559}*{Lemmas~16--18}, $\rho_{E,\ell}(\Gal_\QQ)$ is contained in the normalizer of some Cartan subgroup $C(\ell)$ but is not contained in any Cartan subgroup.   We can thus define a non-trivial character
\[
\varepsilon_\ell \colon \Gal_\QQ \xrightarrow{\rho_{E,\ell}} C^+(\ell)/C(\ell)\xrightarrow{\sim} \{\pm 1\}.
\]
By \cite{MR0387283}*{p.317}, $\varepsilon_\ell$ is unramified at all primes $p \nmid N$ (and hence there are only finitely many possibilities for $\varepsilon_\ell$).  We shall identify $\varepsilon_\ell$ with a Dirichlet character.

For each prime $p\nmid N$, let $E_p$ be the reduction of $E$ modulo $p$ and let $a_p(E)$ be the integer satisfying $|E_p(\FF_p)| = p+1 - a_p(E)$.  Hasse showed that $|a_p(E)|\leq 2\sqrt{p}$.

\begin{lemma} \label{L:divides ap}
Let $\ell$ be a prime as above and $n$ a positive integer such that $\rho_{E,\ell^n}(\Gal_\QQ)$ is contained in the normalizer of a Cartan subgroup.   If $p\nmid N$ is a prime for which  $\varepsilon_\ell(p)=-1$, then $a_p(E)\equiv 0 \pmod{\ell^n}$.
\end{lemma}
\begin{proof}
By assumption there is a Cartan subgroup $C(\ell^n)$ of $\GL_2(\ZZ/\ell^n\ZZ)$ such that $\rho_{E,\ell^n}(\Gal_\QQ) \subseteq C^+(\ell^n)$.  The representation
\[
\Gal_\QQ \xrightarrow{\rho_{E,\ell^n}}  C^+(\ell^n)/C(\ell^n)\xrightarrow{\sim} \{\pm 1\}
\]
agrees with $\varepsilon_\ell$.   If $p\nmid N\ell$, then the representation $\rho_{E,\ell^n}$ is unramified at $p$.  The condition $\varepsilon_\ell(p)=-1$ means that $\rho_{E,\ell^n}(\Frob_p)\subseteq C^+(\ell^n)-C(\ell^n)$.   We have $\tr(A)\equiv 0 \pmod{\ell^n}$ for all $A\in C^+(\ell^n)-C(\ell^n)$ (this can be checked directly from the explicit description of Cartan subgroups in \S\ref{SS:Cartan}).  Therefore,  $a_p(E)\equiv \tr(\rho_{E,\ell^n}(\Frob_p))\equiv 0 \pmod{\ell^n}$.

Now suppose that $p=\ell$.  By \cite{MR0387283}*{p.315 ($c_5$)}, we have $a_\ell(E)\equiv 0 \pmod{\ell}$, so $a_\ell(E)=0$ by the Hasse bound (this uses $\ell\geq 5$).
\end{proof}

Take a prime power $\ell^n$ as in Lemma~\ref{L:divides ap} and suppose that $p\nmid N$ is a prime for which $\varepsilon_\ell(p)=-1$ and $a_p(E)\neq 0$.   Then $\ell^n$ divides $a_p(E)$, and hence $\ell^n \leq 2\sqrt{p}$ by Hasse's bound.   So to bound $\ell^n$ it suffices to effectively choose such a prime $p$.

\begin{lemma}   \label{L:clever}
Let $\varepsilon$ be a non-trivial quadratic Dirichlet character whose conductor divides $N\lcm(N,2)$.  
\begin{romanenum}
\item
Assuming the Generalized Riemann Hypothesis (GRH), there exists an absolute constant $c$ and a prime $p\nmid N$ with
\[
p\leq c (\log N)^2(\log\log(2N))^6
\]
such that $\varepsilon(p)=-1$ and $a_p(E)\neq 0$.

\item
There exists a prime $p\nmid N$ with
\begin{equation*}
p \leq 1152 \cdot N^2 (1+\log \log N) 
\end{equation*}
such that $\varepsilon(p)=-1$ and $a_p(E)\neq 0$.

\end{romanenum}
\end{lemma}
\begin{proof}
For part (i) see the proof of \cite{MR644559}*{Lemme~19} (the proof uses $\varepsilon=\varepsilon_\ell$ but only needs our assumption on the conductor of $\varepsilon$).  It uses an effective version of the Chebotarev density theorem.

We now prove (ii) using the argument in \cite{MR1360773}.   Let $E_2$ be the elliptic curve defined over $\QQ$ obtained by twisting $E_1:=E$ by the character $\varepsilon$.   From our assumption on $\varepsilon$, the curve $E_2$ has good reduction at all primes $p\nmid N$.  For each $p\nmid N$, we have $a_p(E_2)=\varepsilon(p) a_p(E_1)$.   Thus for $p\nmid N$, the condition $\varepsilon(p)=-1$ and $a_p(E)\neq 0$ is equivalent to having $a_p(E_1)\neq a_p(E_2)$.

Let $N_i$ be the conductor of $E_i$ and define $N_i'= N_i \prod_{p|N} p^{d_i(p)},$ where $d_i(p)=0,1$ or $2$ if $E_i$ has good, multiplicative or additive reduction, respectively, at $p$.    Let $M$ be the least common multiple of $N_1'$ and $N_2'$.   By \cite{MR801925}*{\S5 C}, there exists a prime $p\nmid N$ satisfying $p \leq \frac{M}{6} \cdot \prod_{p|M}\big(1+\frac{1}{p}\big)$ and $a_p(E_1)\neq a_p(E_2)$ (this last step uses that $E_1$ and $E_2$ are modular \cite{MR1839918}).   

The integer $M$ has the same prime factors as $N$ and it divides $2^6  3^3 N^2$, so the above prime $p$ is at most $288 \cdot N^2 \prod_{p|N}\big(1+\frac{1}{p}\big)$.  The final bound follows by using $\prod_{p|N}\big(1+\frac{1}{p}\big) \leq 4 (1+\log \log N)$ \cite{MR1360773}*{Lemme~2}.
\end{proof}

Consider the case $n=1$.   If $\rho_{E,\ell}(\Gal_\QQ)\neq \GL_2(\ZZ/\ell\ZZ)$, then by Lemma~\ref{L:clever}(ii) there is a prime $p \leq 1152 \cdot N^2 (1+\log \log N) $ such that $\varepsilon_\ell(p)=-1$ and $a_p(E)\neq 0$.  The prime $\ell$ divides $a_p(E)$ by Lemma~\ref{L:divides ap}, so by the Hasse bound we have
\[
\ell \leq |a_p(E)| \leq 2\sqrt{p} \leq 2\big(1152 \cdot N^2 (1+\log \log N) \big)^{1/2}  < 68 N (1+\log \log N)^{1/2}.
\]
Therefore $\rho_{E,\ell}$ is surjective for all $\ell \geq 68 N (1+\log \log N)^{1/2}$, which is the main result of \cite{MR1360773}.

We now consider higher powers of $n$ while keeping track of divisibilities.   Cojocaru \cite{MR2118760} handles several primes in bounding the product of primes of $\ell$ for which $\rho_{E,\ell}$ is not surjective. 

\begin{prop} \label{P:Serre approach}
Let $E$ be a non-CM elliptic curve over $\QQ$.   Let $N$ be the product of the primes $p$ for which $E$ has bad reduction.  
\begin{romanenum}
\item
Assuming GRH, there is an absolute constant $c$ and a positive integer 
\begin{equation*} \label{E:Serre approach}
M\leq \Big( c (\log N) (\log\log(2N))^3 \Big)^{\omega(N)}
\end{equation*}
such that if $\rho_{E,\ell^n}(\Gal_\QQ)$ is contained in the normalizer of a Cartan subgroup with $\ell>17$ and $\ell\neq 37$, then $\ell^n$ divides $M$.
\item
There is a positive integer 
\[
M\leq  \Big( 68 N (1+\log\log N)^{1/2} \Big)^{\omega(N)}
\]
such that if $\rho_{E,\ell^n}(\Gal_\QQ)$ is contained in the normalizer of a Cartan subgroup with $\ell>17$ and $\ell\neq 37$, then $\ell^n$ divides $M$.
\end{romanenum}
\end{prop}
\begin{proof}
If $N$ is odd, define $N_0=N$, otherwise define $N_0=2N$.  Let $V_1$ be the group of characters $(\ZZ/N_0\ZZ)^\times \to \{\pm 1\}$, which we may view as a vector space of dimension $\omega(N)$ over $\FF_2$.  

We will define a sequence of primes $p_1,\ldots, p_{\omega(N)}$ that are relatively prime to $N$ such that $a_{p_i}(E)\neq 0$ for all $i$ and that for every non-trivial character $\alpha\in V_1$, we have $\alpha(p_i)=-1$ for some $i\in \{1..,\omega(N)\}$.  We proceed inductively on $i=1,\ldots, \omega(N)$.   Choose a non-trivial character $\alpha_i \in V_i$.  Assuming GRH, Lemma~\ref{L:clever}(i) says there is a prime $p_i \nmid N$ satisfying
\begin{equation} \label{E:p with GRH}
p_i\leq c (\log N)^2(\log\log(2N))^6
\end{equation}
such that $\alpha_i(p_i)=-1$ and $a_{p_i}(E)\neq 0$.   Let $V_{i+1}$ be the subspace of $V_i$ consisting of those characters $\varepsilon$ for which $\varepsilon(p_i)=1$.     The space $V_{i+1}$ has dimension $\omega(N)- i$ over $\FF_2$.  Since $V_{\omega(N)+1}=1$, our sequence of primes $p_1,\ldots, p_{\omega(N)}$ has the desired property.

Define the positive integer
\[
M= \prod_{i=1}^{\omega(N)} \big| a_{p_i}(E)\big|.
\]
Consider $\ell$ and $n$ as in the statement of the lemma.   We may view $\varepsilon_\ell$ as a Dirichlet character $(\ZZ/N_0\ZZ)^\times \to \{\pm 1\}$.  Since $\varepsilon_\ell$ is non-trivial, there is some $i\in \{1,\ldots, \omega(N)\}$ such that $\varepsilon_\ell(p_i)= -1$.  By Lemma~\ref{L:divides ap}, we have $a_{p_i}(E)\equiv 0 \pmod{\ell^n}$, and hence $\ell^n$ divides $M$.  It remains to bound $M$.  Using Hasse's bound and (\ref{E:p with GRH}), we obtain
\[
M\leq  \prod_{i=1}^{\omega(N)} 2\Big(c (\log N)^2(\log\log(2N))^6\Big)^{1/2}  = \big( 2c^{1/2}(\log N) (\log\log(2N))^3 \big)^{\omega(N)}.
\]
This completes the proof of (i).

Part (ii) is proved in the exact same way as (i) except that Lemma~\ref{L:clever}(ii) is used to choose the primes $p_1,\ldots, p_{\omega(N)}$.  We then have an appropriate $M$ with
\[
M\leq  \prod_{i=1}^{\omega(N)} 2\Big(1152 \cdot N^2 (1+\log \log N)  \Big)^{1/2}  \leq \Big( 68 N (1+\log\log N)^{1/2} \Big)^{\omega(N)} \qedhere
\]
\end{proof}

\begin{remark} \label{R:Cojocaru}
Let $A(E)$ be the product of the primes $\ell$ for which $\rho_{E,\ell}$ is not surjective.   By Proposition~\ref{P:Serre approach}(ii), we have
\[
A(E)\leq \big(37 \prod_{\ell\leq 17}\ell \big)\cdot  \Big( 68 N (1+\log\log N)^{1/2} \Big)^{\omega(N)}
\]
Theorem~2 of \cite{MR2118760} claims that $A(E)\ll N' (\log\log N')^{1/2}$ where $N'$ is the conductor of $E$, but there seems to be a gap in the proof.   If $\ell$ is a prime greater than 17 and not 37 for which $\rho_{E,\ell}$ is not surjective, then one can use Lemma~\ref{L:clever}(ii), or something similar, to choose a small prime $p\nmid N$ for which $a_p(E)\neq 0$ and $\varepsilon_\ell(p)=-1$.  However, it is not clear why this particular choice of $p$ should work for all such $\ell$.   To make our proof work, we needed to choose at least $\omega(N)$ different primes $p_i$, which is why our bound is essentially the $\omega(N)$-th power of that in \cite{MR2118760}.   
\end{remark}

\section{Masser and W\"ustholz approach} \label{S:MW}

Let $E$ be a non-CM elliptic curve defined over $\QQ$.   Masser and W\"ustholz have shown that $\rho_{E,\ell}(\Gal_\QQ)=\GL_2(\ZZ/\ell\ZZ)$ for all $\ell \geq c \big( \max\{1, h(j_E)\} \big)^\gamma$ where $c$ and $\gamma$ are absolute constants \cite{MR1209248}; they actually prove an analogous results for elliptic curves over an arbitrary number field, but we will continue our focus on curves over $\QQ$.   This theorem was later refined by Masser in \cite{MR1619802}*{Theorem~3} where he proved the following:  there are absolute constants $c$ and $\gamma$ and a positive integer $M\leq c \big(\max\{1,h(j_E)\}\big)^\gamma$ such that $\rho_{E,\ell}(\Gal_\QQ)=\GL_2(\ZZ/\ell\ZZ)$ for all $\ell\nmid M$.

The main tool of Masser and W\"ustholz is an effective bound of the minimal possible degree of an isogeny between two isogenous abelian varieties.  We will use the following variant \cite{MR1619802}*{p.190 Theorem D}.  For a principally polarized abelian variety $A$ over a number field, let $h(A)$ denote its (absolute, logarithmic, semistable) Faltings height.

\begin{thm} \label{T:MW multiplicative}
Let $r$ and $D$ be positive integers.  Then there are constants $c$ and $\kappa$, depending only on $r$ and $D$, with the following property.  Suppose that $A$ is a principally polarized abelian variety of dimension $r$ defined over a number field $k$.  Suppose further that $A$ is isomorphic over $k$ to a product $A_1^{e_1}\times \cdots \times A_t^{e_t}$, where $A_1,\ldots, A_t$ are simple, mutually non-isogenous, and have trivial endomorphism rings.  Then there is a positive integer
\[
b(k,A,D) \leq c \big(\max\{[k:\QQ], h(A)\} \big)^\kappa
\]
such that if $A^*$ is an abelian variety, defined over an extension $K$ of $k$ of relative degree at most $D$, that is isogenous over $K$ to $A$, then there is an isogeny over $K$ from $A^*$ to $A$ whose degree divides $b(k,A,D)$.
\end{thm}

Our application of this theorem is the following.   The proof is the same as \cite{MR1209248}*{Lemma~3.2} except using the above refined theorem and considering higher prime powers.
\begin{prop} \label{P:MW approach}
There are absolute constants $c$ and $\kappa$ with the following property.   Suppose that $E$ is a non-CM elliptic curve defined over $\QQ$.  Then there is a positive integer
\[
M \leq c \big( \max\{1 , h(j_E) \} \big)^\kappa
\]
such that if $\rho_{E,\ell^n}(\Gal_\QQ)$ is contained in the normalizer of a Cartan subgroup with $\ell\geq 7$, then $\ell^n$ divides $M$.
\end{prop}
\begin{proof}
Suppose that $\ell \geq 7$ is a prime for which $\rho_{E,\ell^n}(\Gal_\QQ)$ is contained in the normalizer of a Cartan subgroup.  There is a quadratic extension $K$ of $\QQ$ such that $\rho_{E,\ell^n}(\Gal(\Qbar/K))$ is contained in a Cartan subgroup of $\GL_2(\ZZ/\ell^n\ZZ)$.  By \cite{MR644559}*{Lemme 18'}, the image of $\rho_{E,\ell}(\Gal_\QQ)$ in $\PGL_2(\ZZ/\ell\ZZ)$ contains an element of order at least $(\ell-1)/4$, and in particular of order greater than $1$ since $\ell\geq 7$.   Therefore, $\rho_{E,\ell^n}(\Gal(\Qbar/K))$ contains an element $\alpha$ for which $\alpha \pmod{\ell}$ is not a scalar matrix.   

Let $\Gamma$ be the subgroup of $A:=E\times E$ consisting of pairs of points $(x, \alpha x)$ with $x\in E[\ell^n]$.   The group $\Gamma$ is defined over $K$, since an arbitrary $\sigma\in \Gal(\Qbar/K)$ takes $(x, \alpha x)$ to $(\rho_{E,\ell^n}(\sigma) x , \rho_{E,\ell^n}(\sigma) \alpha x)$, which is $(y, \alpha y)$ for $y=\rho_{E,\ell^n}(\sigma) x$ (we have used that $\rho_{E,\ell^n}(\Gal(\Qbar/K))$ is commutative).  Therefore, the abelian variety $A^* := A/\Gamma$ and the natural isogeny $\psi \colon A \to A^*$ are both defined over $K$.  By Theorem~\ref{T:MW multiplicative}, there is an isogeny $\phi\colon A^* \to A$ defined over $K$ whose degree divides $b(\QQ, E\times E, 2)$.   The composition $\varepsilon:=\phi\circ \psi$ is an endomorphism of $A=E\times E$.   Since $E$ is non-CM, there are integers $a,b,c,d\in \ZZ$ such that $\varepsilon(x,y) =(a x + by, cx+d y)$.  Since $\Gamma\subseteq \ker(\varepsilon)$, we have $ax + b \alpha x = 0$ and $cx+d \alpha x$ for all $x\in E[\ell^n]$.   Equivalently, $a + b \alpha\equiv 0$ and $c + d\alpha \equiv 0 \pmod{\ell^n}$.  Since $\alpha \pmod{\ell}$ is non-scalar, we deduce that $\ell^n$ divides $a$, $b$, $c$ and $d$.  Therefore, $\ell^{4n}$ divides $\deg \varepsilon =  (ad-bc)^2$.  Since $\deg \phi = |\Gamma| = \ell^{2n}$ and $\deg \varepsilon = \deg\phi \cdot \deg \psi$, we find that $\ell^{2n}$ divides $\deg \psi$ and hence also $b(\QQ,E\times E,2)$.    

So we shall take $M=b(\QQ,E\times E,2)$, and it remain to prove that $M$ has the desired upper bound.   Since $b(\QQ,E\times E,2) \leq c(\max\{1, h(E\times E)\} )^\kappa$ for absolute constants $c$ and $\kappa$, it suffices to note that $h(E\times E)=2 h(E)$ and that $h(E)\ll \max\{1,h(j_E)\}$ \cite{MR861979}*{Proposition 2.1}.      
\end{proof}

\section{Proof of Proposition~\ref{P:put together} and Theorem~\ref{T:main}} \label{S:main proof}

We first prove the following lemma to deal with the small primes.

\begin{lemma}   \label{L:small primes}
For each prime $\ell$, there is an integer $e(\ell)\geq 1$, depending only on $\ell$, such that $\rho_{E,\ell^\infty}(\Gal(\Qbar/\QQ^\cyc))$ contains the group $\{A\in \SL_2(\ZZ_\ell) : A \equiv I \pmod{\ell^{e(\ell)}}\}$ for all non-CM elliptic curves $E$ over $\QQ$.
\end{lemma}
\begin{proof}
Fix a prime $\ell$.   For each subgroup $H$ of $\GL_2(\ZZ/\ell^n\ZZ)$ with $\det(H)=(\ZZ/\ell^n\ZZ)^\times$, we can define a modular curve $X_H$ over $\QQ$, which comes with a morphism $\pi\colon X_H \to \PP^1_\QQ$ to the $j$-line \cite{MR0337993}*{IV-3}.     If $E$ is an elliptic curve over $\QQ$ with $\rho_{E,\ell^n}(\Gal_\QQ)\subseteq \GL_2(\ZZ/\ell^n\ZZ)$ conjugate to a subgroup of $H$, then there is a rational point $P\in X_H(\QQ)$ such that $\pi(P)$ equals the $j$-invariant of $E$.

Let $\Gamma_H$ be the congruence subgroup of $\SL_2(\ZZ)$ consisting of those $A\in \SL_2(\ZZ)$ for which $A$ modulo $\ell^n$ belongs to $H$.  The genus of $X_H$ is equal to the genus of $\Gamma_H $ (which is the genus of the Riemann surface obtained by the usual quotient of the upper-half plane by the group $\Gamma_H$ acting via linear fractional transformations and adding cusps).  There are only finitely many congruence subgroups of $\SL_2(\ZZ)$ of genus  $0$ or $1$ \cite{MR0384698}.  Hence, there exists an integer $f=f(\ell)\geq 1$ such that if $H$ is a subgroup of $\GL_2(\ZZ/\ell^{f+2}\ZZ)$ for which $X_H$ has genus at most $1$, then $H\supseteq \{ A \in \SL_2(\ZZ/\ell^{f+2}\ZZ) : A\equiv I \pmod{\ell^f} \}$.  Apply Faltings' theorem (originally, Mordell's conjecture) to the modular curves $X_H$ with groups $H\subseteq\GL_2(\ZZ/\ell^{f+2}\ZZ)$ not containing $\{ A \in \SL_2(\ZZ/\ell^{f+2}\ZZ) : A\equiv I \pmod{\ell^f} \}$, we find that there is a \emph{finite} set $J\subseteq \QQ$ such $\rho_{E,\ell^{f+2}}(\Gal_\QQ)\supseteq \{ A \in \SL_2(\ZZ/\ell^{f+2}\ZZ) : A\equiv I \pmod{\ell^f} \}$ for all elliptic curves $E/\QQ$ with $j(E)\notin J$.   If $E/\QQ$ satisfies $j(E)\not\in J$, then Lemma~\ref{L:filtration basics}(v) implies that $\rho_{E,\ell^\infty}(\Gal(\Qbar/\QQ^\cyc))\supseteq \{A \in \SL_2(\ZZ_\ell): A \equiv I \pmod{\ell^{2f+2}} \}$  (since $\QQ^\cyc$ is the maximal abelian extension of $\QQ$, $\rho_{E,\ell^\infty}(\Gal(\Qbar/\QQ^\cyc))$ is the commutator subgroup of $\rho_{E,\ell^\infty}(\Gal_\QQ)$).    The desired $e(\ell)$ is chosen such that $e(\ell)\geq 2f+2$ and such that the lemma holds for the finite number of excluded $\Qbar$-isogeny classes of non-CM elliptic curves over $\QQ$.
\end{proof}

\begin{proof}[Proof of Proposition~\ref{P:put together}]
Define the set $Q=\{ \ell : \ell \leq 17\} \cup \{37\}$ and let $\mathcal{P}$ be the (finite) set of primes $\ell\notin Q$ for which $\rho_{E,\ell}$ is not surjective.   Define the number field $K= \bigcup_{\ell \in \mathcal{P}\cup Q} \QQ( E[\ell])$; it is a Galois extension of $\QQ$.      Define the group $S=\rho_{E}(\Gal(\Qbar/K\QQ^\cyc))$.   Since $\det\circ \rho_E\colon \Gal_\QQ\to \Zhat^\times$ is the cyclotomic character and is thus surjective,  we have
\begin{align} \label{E:GL2 bound to nilpotent}
[\GL_2(\Zhat):\rho_E(\Gal_\QQ)] & = [\SL_2(\Zhat):\rho_E(\Gal(\Qbar/\QQ^\cyc))] =\frac{1}{[K: K \cap \QQ^\cyc]} [\SL_2(\Zhat) : S].
\end{align}
So it remains to bound $[\SL_2(\Zhat) : S]$.  For each prime $\ell$, let $S_\ell$ be the image of $S$ under the projection to $\SL_2(\ZZ_\ell)$.  We now describe the groups $S_\ell$.

\noindent\textbf{Case 1:} Let $\ell$ be a prime in $\mathcal{P}$.

Define $f(\ell) = \ord_\ell(M)+1$.   Since $\ell^{f(\ell)} \nmid M$, the group $\rho_{E,\ell^{f(\ell)}}(\Gal_\QQ)$ is not contained in the normalizer of a Cartan subgroup.  By Proposition~\ref{P:dichotomy main}, we must have $\rho_{E,\ell^\infty}(\Gal_\QQ) \supseteq I+\ell^{4f(\ell)}M_2(\ZZ_\ell)$   and hence  $\rho_{E,\ell^\infty}(\Gal(\Qbar/\QQ^\cyc))$ contains $\{A \in \SL_2(\ZZ_\ell) : A \equiv I \pmod{\ell^{4f(\ell)}} \}$ by Lemma~\ref{L:dichotomy supplement}.  Therefore,
\[
[\SL_2(\ZZ_\ell): S_\ell ] \leq  \big(\ell^{4f(\ell)}\big)^3 \ell^{\ord_\ell([K:K\cap \QQ^\cyc])} = \big(\ell^{\ord_\ell(M)+1}\big)^{12} \ell^{\ord_\ell([K:K\cap \QQ^\cyc])} . 
\]
The prime $\ell$ divides $M$, so $[\SL_2(\ZZ/\ZZ_\ell): S_\ell ] \leq  \ell^{\ord_\ell([K:K\cap \QQ^\cyc])}  \big(\ell^{\ord_\ell(M)}\big)^{24}$.

\noindent\textbf{Case 2:} Let $\ell$ be a prime in $Q$.

By Lemma~\ref{L:small primes}, there is an integer $e(\ell)\geq 1$, not depending on $E$, such that $\rho_{E,\ell^\infty}(\Gal(\Qbar/\QQ^\cyc))$ contains $\{A \in \SL_2(\ZZ_\ell) : A \equiv I \pmod{\ell^{e(\ell)}} \}$.  Therefore,
\[
[\SL_2(\ZZ/\ZZ_\ell): S_\ell ] \leq  \ell^{3e(\ell)} \ell^{\ord_\ell([K:K\cap \QQ^\cyc])}.
\]
\noindent\textbf{Case 3:} Let $\ell$ be a prime not in $\mathcal{P}\cup Q$.

We claim that $S_\ell = \SL_2(\ZZ_\ell)$.  Since $\ell$ is not in $\mathcal{P}$, we have $\rho_{E,\ell}(\Gal_\QQ)=\GL_2(\ZZ/\ell\ZZ)$.   No composition series of the Galois group $\Gal(K/\QQ)$ contains the simple group $\SL_2(\ZZ/\ell\ZZ)/\{\pm I\}$ (this follows from the description of the sets ``$\text{Occ}(\GL_2(\ZZ_p))$'' in \cite{MR1484415}*{IV-25}).  Therefore, the group $\rho_{E,\ell}(\Gal(\Qbar/K))$, and hence also $H:=\rho_{E,\ell}(\Gal(\Qbar/K\QQ^{\cyc}))\subseteq \SL_2(\ZZ/\ell\ZZ)$, contains the group $\SL_2(\ZZ/\ell\ZZ)/\{\pm I\}$ in its composition series.   The image of $H$ under the quotient $\SL_2(\ZZ/\ell\ZZ)\to \SL_2(\ZZ/\ell\ZZ)/\{ \pm I\}$ is thus surjective, so $H=\SL_2(\ZZ/\ell\ZZ)$ by \cite{MR1484415}*{IV-23~Lemma~2}.  By \cite{MR1484415}*{IV-23~Lemma~3}, we have $\rho_{E,\ell^\infty}(\Gal(\Qbar/K\QQ^{\cyc}))=\SL_2(\ZZ_\ell)$.\\

Returning to the group $S$, we may view $S$ as a closed subgroup of $\prod_\ell S_\ell \subseteq \prod_\ell \SL_2(\ZZ_\ell) = \SL_2(\Zhat)$.   We will show that $S= \prod_{\ell} S_\ell$.    For each finite set of primes $B$, let $p_B\colon S\to S_B:=\prod_{\ell \in B} S_\ell$ be the projection.    We have $p_B(S)=S_B$ for $B= \mathcal{P} \cup Q$ since the groups $S_\ell$ with $\ell \in B$ are pro-$\ell$ by our choice of $K$ (the group $p_{B}(S)$ is pro-nilpotent so is a product of its Sylow subgroups).   Now assume that $B$ is a set of primes containing $\mathcal{P} \cup Q$ for which $p_B(S)=S_B$ and  fix a prime $\ell_0\notin B$.  We claim that $p_{B \cup \{\ell_0\}}(S)=S_{B \cup \{\ell_0\}}$.   Define $N=p_{B\cup \{\ell_0\}}\big(p_{\{\ell_0\}}^{-1}(1)\big)$ and $N'=p_{B\cup \{\ell_0\}}\big(p_{B}^{-1}(1)\big)$ which are closed normal subgroups of $p_{B\cup\{\ell_0\}}(S)$ that we may also view as subgroups of $S_B$ and $S_{\ell_0}$, respectively.   By Goursat's lemma, we find that the image of $S$ in $S_B/N \times S_{\ell_0}/N'$ is the graph of an isomorphism $S_B/N \cong S_{\ell_0}/N'$ of profinite groups (cf.~\cite{MR0457455}*{\S2}, it is clear that the isomorphism and its inverse are continuous).     The groups $S_{\ell_0}=\SL_2(\ZZ_{\ell_0})$ and $S_B=\prod_{\ell\in B} S_\ell$ have no isomorphic non-abelian simple groups in their composition series (again see \cite{MR1484415}*{IV-25} and note that $\ell$-groups are solvable), so the isomorphism $S_B/N \cong S_{\ell_0}/N'$ is an isomorphism of solvable groups.    However, there are no non-trivial abelian quotients of $\SL_2(\ZZ_{\ell_0})$ \cite{MR2118760}*{Appendix Corollary~7}, so $N=S_B$ and $N'=S_{\ell_0}$.  Thus $p_{B\cup \{\ell_0\}}(S)=S_B \times S_{\ell_0} = S_{B\cup \{\ell_0\}}$ as claimed.  By induction, we have $p_B(S)=S_B$ for any sets of primes $B$ containing $\mathcal{P} \cup Q$, and since $S$ is profinite we deduce that $S=\prod_{\ell} S_\ell$.

Returning to the index and using the three cases, we have
\[
[\SL_2(\Zhat) : S] = \prod_{\ell} [\SL_2(\ZZ_\ell): S_\ell ] \leq [K:K\cap \QQ^\cyc] \prod_{\ell\in Q} \ell^{3e(\ell)}  \big(\prod_{\ell\in \mathcal{P}} \ell^{\ord_\ell(M)}\big)^{24}. 
\]
By (\ref{E:GL2 bound to nilpotent}), we conclude that $[\GL_2(\Zhat):\rho_E(\Gal_\QQ)]  \leq  C M^{24}$ where the constant $C:=\prod_{\ell\in Q} \ell^{3e(\ell)}$ is absolute.  
\end{proof}

Theorem~\ref{T:main} is now easy.  Part (i) follows immediately by combining Proposition~\ref{P:put together} and Proposition~\ref{P:MW approach}.   Parts (ii) and (iii) follow by combining Proposition~\ref{P:put together} with Proposition~\ref{P:Serre approach}(ii) and (i), respectively.

\begin{remark}\label{R:uniform index}
Suppose that there exists an absolute constant $c$ such that $\rho_{E,\ell}(\Gal_\QQ)=\GL_2(\ZZ/\ell\ZZ)$ for all non-CM elliptic curves $E$ over $\QQ$ and primes $\ell>c$.   We may assume that $c$ is the smallest such constant.  For each prime $\ell \leq c$, let $e(\ell)$ be a positive integer as in Lemma~\ref{L:small primes}.   The integer $M=\prod_{\ell \leq c} \ell^{e(\ell)}$ satisfies the condition of Proposition~\ref{P:put together}, so $[ \GL_2(\Zhat) : \rho_E(\Gal_\QQ) ] \leq C M^{24}$ where $C$ and $M$ are absolute constants. 
\end{remark}

\section{Lang-Trotter constants} \label{S:LT}

Let $E$ be an non-CM elliptic curve over $\QQ$.   The predicted constant $C_{E,r}$ of Conjecture~\ref{C:LT} is 
\[
C_{E,r} = \frac{2}{\pi}  \lim_{m \to \infty}  m \frac{|\{ A \in \rho_{E,m}(\Gal_\QQ) : \tr(A)\equiv r \pmod{m} \}|}{|\rho_{E,m}(\Gal_\QQ)|}.
\]
where the limit is over all natural numbers $m$ ordered by divisibility \cite{MR0568299}*{I~\S4}.  Serre's open image theorem is used to prove that this limit converges.  The constant $C_{E,r}$ for CM elliptic curves is not studied in \cite{MR0568299} but can be found in \cite{MR2534114}*{\S2.2}.     To control the constants $C_{E,r}$ for non-CM $E/\QQ$, we will need the following easy lemma.

\begin{lemma} \label{L:LT inequality}  Let $E$ be a non-CM elliptic curve over $\QQ$.  Then $C_{E,r} \leq  [\GL_2(\Zhat): \rho_E(\Gal_\QQ)] \cdot C_r$ where $C_r$ is the constant of Theorem~\ref{T:DP}.
\end{lemma}
\begin{proof}
We have 
\begin{align*}
C_{E,r} &\leq \frac{2}{\pi}  \limsup_{m\to \infty} \, m \frac{|\{A \in \GL_2(\ZZ/m\ZZ) : \tr(A)\equiv r \pmod{m}\}|}{|\rho_{E,m}(\Gal_\QQ)|} \\
&=  [\GL_2(\Zhat): \rho_E(\Gal_\QQ)] \cdot \frac{2}{\pi}  \limsup_{m\to \infty}  \, m \frac{|\{A \in \GL_2(\ZZ/m\ZZ) : \tr(A)\equiv r \pmod{m}\}|}{|\GL_2(\ZZ/m\ZZ)|}
\end{align*}
and this last line equals $ [\GL_2(\Zhat): \rho_E(\Gal_\QQ)] \cdot C_r$ by the computations in \cite{MR0568299}*{I~\S4} (with $M=1$).
\end{proof}

We will now prove Theorem~\ref{T:LT}. 

 Let $\calS(A,B)$ be the set of $(a,b)\in \calF(A,B)$ for which $[\GL_2(\Zhat): \rho_{E(a,b)}(\Gal_\QQ)] =2$ (such elliptic curves are called \emph{Serre curves}) and define $\calN(A,B)=\calF(A,B)-\calS(A,B)$.    By \cite{MR2534114}*{Theorem 10} with $k=1$, we have
\begin{equation} \label{E:LT 1}
\frac{1}{|\calF(A,B)|} \sum_{(a,b) \in \calS(A,B)} C_{E(a,b),r} = C_r  + O\Big(\frac{1}{A} + \frac{(\log B)^{1/2}(\log A)^{7/2}}{B^{1/2}}\Big)
\end{equation}
where the implicit constant depends only on $r$.  

Now fix a pair $(a,b)\in\calN(A,B)$.  If $E(a,b)$ is CM, then we have $C_{E(a,b),r}\ll_r 1$ as noted in \cite{MR2534114}*{Lemma~23} (the key point being that there are only finitely many $\Qbar$-isogeny classes of CM curves over $\QQ$).  If $E(a,b)$ is non-CM, then by Lemma~\ref{L:LT inequality} and Theorem~\ref{T:main}(i), we have
\[
C_{E(a,b),r} \ll_r  \big( \max\{1, h(j_{E(a,b)})\} \big)^\gamma
\]
where $\gamma$ is an absolute constant.   Since $(a,b) \in\ZZ^2$ with $|a|\leq A$ and $|b|\leq B$, we have 
\[
h(j_{E(a,b)}) = h([1728(4a)^3, -16(4a^3+27b^2)]) \ll \log(\max\{A, B\}) \leq \log(AB)
\]
and hence $C_{E(a,b),r} \ll_r \log^\gamma(AB)$.  Therefore,
\begin{align*}
&\frac{1}{|\calF(A,B)|}\sum_{(a,b) \in \calN(A,B)} C_{E(a,b), r} \ll_r \frac{|\calN(A,B)|}{|\calF(A,B)|} \log^\gamma(AB).
\end{align*}
Jones \cite{Jones-AAECASC}*{Theorem 4} has shown that there is an absolute constant $\beta>0$ such that
\[
\frac{|\calN(A,B)|}{|\calF(A,B)|} \ll \frac{\log^\beta(\min\{A, B\})}{\sqrt{\min\{A, B\}}},
\]
so
\begin{equation} \label{E:LT 2}
\frac{1}{|\calF(A,B)|} \sum_{(a,b) \in \calN(A,B)} C_{E(a,b), r} \ll_r \frac{\log^{\beta+\gamma}(AB)}{\sqrt{\min\{A, B\}}}.
\end{equation}
Theorem~\ref{T:LT} follows by adding (\ref{E:LT 1}) and (\ref{E:LT 2}) together, and taking $\delta =\max\{4, \beta+\gamma\}$.

\end{document}